\documentclass[12pt, leqno]{amsart}
\usepackage{amsmath}
\usepackage{amssymb}
\usepackage{amsfonts}
\usepackage{graphicx}
\usepackage{amsthm}
\usepackage{enumerate}
\usepackage{lscape}
\usepackage{dsfont}
\usepackage{color}
\usepackage{mathtools}

\newtheorem{theor}{Theorem}

\newtheorem{remark}[theor]{Remark}

\title{Stenzel's Ricci--flat K\"ahler metrics are not projectively induced}
\author[M. Zedda]{Michela Zedda}
\date{\today}
\address{Dipartimento di Scienze Matematiche, Fisiche e Informatiche\\ Universit\`a di Parma} 
 \email{michela.zedda@unipr.it}
\thanks{This research has been financially supported by the Programme ``FIL-Quota Incentivante'' of University of Parma and co-sponsored by Fondazione Cariparma, and by the group G.N.S.A.G.A. of I.N.d.A.M} 

\begin{document}
\maketitle
\begin{abstract}
We investigate the existence of a holomorphic and isometric immersion in the complex projective space for the complete Ricci--flat K\"ahler metrics constructed by M. B. Stenzel in \cite{stenzel} on the cotangent bundle of a compact, rank one, globally symmetric space. 
\end{abstract}
\section{Introduction}
The problem of classifying Ricci--flat projectively induced metrics is an interesting and long lasting problem in K\"ahler geometry. A K\"ahler metric $g$ on a complex manifold $M$ is projectively induced if $(M,g)$ admits a holomorphic and isometric immersion into the complex projective space $\mathds C{\rm P}^N$, of dimension $N\leq +\infty$, endowed with its Fubini--Study metric $g_{FS}$, i.e. if there exists a holomorphic map $f\!:M\rightarrow \mathds C{\rm P}^N$ such that $f^*g_{FS}=g$. 

Observe that a Ricci-flat projectively induced metric is forced to be not compact, due to a result of D. Hulin \cite{hulin}, stating that a projectively induced K\"ahler--Einstein metric on a compact manifold must have positive scalar curvature.
On the other hand, the flat metric on $\mathds C^n$ is well-known to be projectively induced in $\mathds C{\rm P}^\infty$ (see e.g. \cite{calabi,book}). Since Ricci flat metrics arise as solutions of a Monge-Ampere equation, one should expect that the flat case is the only “algebraic” (namely projectively induced) one. This conjecture was firstly stated in \cite{lsz} by A. Loi, F. Salis and F. Zuddas, where they verify it for metrics admitting a radial K\"ahler potential. 
The first example of Ricci--flat (non-flat) K\"ahler metric constructed on a non-compact manifold is the Taub--NUT metric described by C. LeBrun in \cite{lebrun}. This is a $1$--parameter family of complete K\"ahler metrics on $\mathds C^2$ defined by the K\"ahler potential $\Phi_m(u,v)=u^2+v^2+m(u^4+v^4)$, where $u$ and $v$ are implicitly given by $|z_1|=e^{m(u^2+v^2)}u$, $|z_2|=e^{m(v^2-u^2)}v$. One can prove (see \cite{taubnut}) that for $m>\frac12$ the Taub--NUT metric is not projectively induced. Actually, with the same techniques used in  \cite{taubnut}, one can prove that it is not also for smaller values of the parameter. Although, it is hard to prove it for values of $m$ approaching to $0$, coherently with the fact that for $m=0$ the Taub--NUT metric reduces to be the flat metric on $\mathds C^2$. 
In \cite{tohoku} the author of the present paper in collaboration with A. Loi and F. Zuddas, proved that the Ricci flat Calabi's metrics constructed in \cite{calabi79} on holomorphic line bundles over compact K\"ahler--Einstein manifolds are not projectively induced. As a byproduct they solve a conjecture addressed in \cite{lsz} by proving that {\em any positive multiple} of the Eguchi-Hanson metric on the blow-up of $\mathds C^2$ at the origin is not projectively induced. 
Observe that rescaling a Ricci--flat K\"ahler metric $g$ by a positive factor $c$, gives a Ricci-flat K\"ahler metric $cg$ which could in principle be projectively induced even when $g$ is not. For example, the Bergman metric on bounded symmetric domains of $\mathds C^n$ is not projectively induced when rescaled by small values of $c$, but it is for large values (see \cite{cartan} for more details).

In this paper we are interested in studying the complete Ricci--flat K\"ahler metrics constructed by Stenzel on the cotangent bundle of a compact, rank one, globally symmetric space. These metrics arise in several constructions of Calabi--Yau manifolds (see e.g. \cite{doicu,spotti}) and a generalization to rank two symmetric spaces appeared recently in \cite{rank2}. 
An explicit K\"ahler potential for (almost) all Stenzel metrics, needed in the proof of our result, has been given by T.-C. Lee in \cite{lee}.
In our context, these metrics are of particular interest because their K\"ahler potentials are not rotation invariant (i.e. depending only on the modules of the variables) as those studied before. Further, they are constructed over homogeneous manifolds, which are projectively induced when rescaled by a positive factor (see e.g. \cite[Ch. 3]{book}). In this sense, it is emblematic the behaviour of Cartan--Hartogs domains, which are nonhomogeneous K\"ahler--Einstein manifolds with negative scalar curvature constructed over a bounded symmetric domain, that are projectively induced when rescaled by a large enough factor (see \cite{cartan}). Although, the Ricci--flat case appears to be less flexible, as for example  Calabi's Ricci--flat K\"ahler metrics constructed over flag manifolds are not projectively induced even when rescaled \cite{tohoku}. 

We prove the following theorem, where we give evidence of the above mentioned conjecture for Stenzel's metrics constructed on the complexification of $\mathds C{\rm P}^n$ and $\mathds H{\rm P}^n$:
\begin{theor}\label{miith}
Stenzel's Ricci--flat complete K\"ahler metrics on the complexification of $\mathds C{\rm P}^n$ and $\mathds H{\rm P}^n$ are not projectively induced for any $n> 1$. The same result holds when the metrics are rescaled by a constant $0<c\leq 1$.
\end{theor}

The results are obtained by applying Calabi's criterion, which we describe in the next section.
 It is worth pointing out that the problem of verifying through Calabi's techniques that a metric $g$ is not projectively induced when rescaled by a positive constant $c$, becomes harder as $c$ grows.
In Remark \ref{hardpart} we describe the general problem for the metric on the complexification of $\mathds C{\rm P}^n$, giving an idea of the computational complexity.

The paper is organized as follows. In the next section we set the notations and summarize Calabi's criterion \cite{calabi} for K\"ahler immersions into the complex projective space. The third and last section is devoted to a description of  Stenzel's metrics on the complexification of $\mathds C{\rm P}^n$ and $\mathds H{\rm P}^n$, and to proving our result.

\section{Calabi's criterion}\label{calcrit}
We refer the reader to \cite{book} and references therein for a more detailed overview on the subject.

We denote by $\mathds{C}\mathrm{P}^N$ the complex projective space of dimension $N\leq \infty$, endowed with its Fubini-Study metric $g_{FS}$. Consider homogeneous coordinates $[Z_0,\dots,Z_N]$ and define in the usual way affine coordinates $z_1,\dots, z_N$ on $U_0=\{Z_0\neq 0\}$ by $z_j=Z_j/Z_0$. A K\"ahler potential for the Fubini--Study metric on $U_0$ is given by:
$$
\varphi_{FS}(z)=\log\left(1+\sum_{j=1}^N|z_j|^2\right).
$$

Let $(M,g)$ be a real analytic K\"ahler manifold of dimension $n$ and fix a coordinates system $(z_1,\dots, z_n)$ in a neighborhood $U$ of a point $p\in M$. Let also $\varphi\!: U\rightarrow \mathds{R}$ be a K\"ahler potential for $g$ on $U$, i.e.:
$$
g_{j\bar k}(z)=\frac{\partial^2\varphi(z)}{\partial z_j\partial \bar z_k}.
$$
Observe that it is not restrictive in our context to assume that $g$ is real analytic, since the pull--back through a holomorphic map of the real analytic Fubini--Study metric is forced to be real analytic itself.
Denote by $\tilde\varphi\!: W\rightarrow \mathds{R}$, $\tilde\varphi(z,\bar z)=\varphi(z)$, the analytic extension of $\varphi$ on a neighborhood $W$ of the diagonal in $U\times \bar U$. The {\em diastasis function} ${\rm D}(z,w)$ is defined by:
\begin{equation}\label{diast}
{\rm D}(z,w):=\tilde\varphi(z,\bar z)+\tilde\varphi(w,\bar w)-\tilde\varphi(z,\bar w)-\tilde\varphi(w,\bar z).
\end{equation}
Observe that it follows easily from the definition that once one of its two entries is fixed, the diastasis is a K\"ahler potential for $g$. In particular, we denote ${\rm D}_0(z):={\rm D}(z,0)$. The local Calabi's criterion for K\"ahler immersions into $\mathds C{\rm P}^N$ can be expressed as follows:
\begin{theor}[Calabi's criterion \cite{calabi}]\label{localcrit}
Let $(M,g)$ be a K\"ahler manifold. A neighborhood of a point $p\in M$ admits a K\"ahler immersion into $\mathds C{\rm P}^N$ if and only if the $\infty\times \infty$ hermitian matrix of coefficients $(a_{jk})$ defined by:
\begin{equation}\label{expansion}
e^{D_0(z)}-1=\sum_{j,k=0}^\infty a_{jk}z^{m_j}\bar z^{m_k},
\end{equation}
is positive semidefinite of rank at most $N$.
\end{theor}
Here we are using a multi-index notation $z^{m_j}:=z_1^{m_{j,1}}\cdots z_n^{m_{j,n}}$, where the $n$-tuples $m_j=(m_{j,1},\dots, m_{j,n})$ satisfies $j<k$ when $|m_j|<|m_k|$, and those with the same module follow a lexicographic order.

\begin{remark}\label{adaptedcriterion}\rm
Let $(M,g)$ be a K\"ahler manifold, $(z_1,\dots,z_n)$ be local coordinates in a neighborhood of a point $p\in M$ and let $\varphi$ be a K\"ahler potential centered at $p$. Observe that:
$$
\frac{\partial^{2k}e^{D_0(z)}}{\partial z_1^k\partial \bar z_1^k}|_0<0,\quad k\geq 1,
$$
are (up to multiplication by a positive constant) elements on the diagonal of the matrix $(a_{jk})$ in \eqref{expansion}, thus if one of them is negative, $g$ is not projectively induced.
\end{remark}

\section{Proof of Theorem \ref{miith}}
We consider the explicit formulas given by T.-C. Lee  in \cite{lee} for the complete Ricci--flat K\"ahler metrics on the complexification of  $\mathds C{\rm P}^n$ and $\mathds H{\rm P}^n$. We refer the reader to \cite{stenzel} for the general construction and the proof of their completeness.

\subsection{Complete Ricci--flat K\"ahler metric on the complexification of $\mathds C{\rm P}^n$} Consider homogeneous coordinates $(z,w)=(z_0,\dots, z_n,w_0,\dots, w_n)$ on $\mathds C{\rm P}^n\times \mathds C{\rm P}^n$ and let 
$$
{\rm M}_{II}^{2n}:=\mathds C{\rm P}^n\times \mathds C{\rm P}^n\setminus Q_\infty,
$$
 where $Q_\infty:=\{(z,w)\in \mathds C{\rm P}^n\times \mathds C{\rm P}^n|\, \sum_{j=0}^nz_jw_j=0\}$. Fix affine coordinates $(1,z_1,\dots, z_n,1,w_1\dots, w_n)$ and consider the K\"ahler metric $g_{II}$ on ${\rm M}_{II}^{2n}$ defined by the K\"ahler potential $f(\mathcal N)$, where:
\begin{equation}\label{n2}
\mathcal N:=\frac{\left(1+\sum_{j=1}^n|z_j|^2\right)\left(1+\sum_{j=1}^n|w_j|^2\right)}{|1+\sum_{j=1}^nz_jw_j|^2},
\end{equation}
and $f$ is a solution to:
\begin{equation}\label{eqmii}
\left(2\mathcal N-1\right)\mathcal N^{n-1}\left(f'\right)^{2n}+2\left(\mathcal N-1\right)\mathcal N^n\left(f'\right)^{2n-1}f''=1.
\end{equation}
Using \eqref{eqmii}, we get $f'(\mathcal N)=\mathcal N^{-1/2}$ (see \cite[p.321]{lee}), from which follows  that:
\begin{equation}\label{fexpansion2}
f'(1)=1,\quad f''(1)=-\frac12,\quad f'''(1)=\frac34,\quad f^{iv}(1)=-\frac{15}8.
\end{equation}
We can now prove the first part of Theorem \ref{miith}.
\begin{proof}[Proof of Theorem \ref{miith} for $M_{II}^{2n}$]
Let $f(\mathcal N)$ be defined by \eqref{n2} and \eqref{fexpansion2}. By Remark \ref{adaptedcriterion}, it is enough to prove that for $0<c\leq 1$:
$$
\frac{\partial^{8}e^{cD_0(z)}}{\partial z_1^4 \partial \bar z_1^4}|_0<0,
$$
where by definition \eqref{diast}:
$$
D_0(z)=f(\mathcal N)+f(1)-f\left(\frac{1}{1+\sum_{j=1}^nz_jw_j}\right)-f\left(\frac{1}{1+\sum_{j=1}^n\bar z_j\bar w_j}\right).
$$
Since the derivative is evaluated at the origin and we are deriving only with respect to $z_1$, $\bar z_1$, let us restrict ourselves to $z_2=\dots=z_n=0$,  $w_1=\dots=w_n=0$, i.e. $D_0(z)=f(1+|z_1|^2)-f(1)$. 
Compute:
\begin{equation}
\frac{\partial e^{cD_0}}{\partial z_1}=ce^{cD_0} f' \bar z_1; \nonumber
\end{equation}
\begin{equation}
\frac{\partial^2 e^{cD_0}}{\partial z_1^2}=ce^{cD_0}(cf'^2+f'') \bar z_1^2,\nonumber
\end{equation}

\begin{equation}
\frac{\partial^3 e^{cD_0}}{\partial z_1^3}=ce^{cD_0}(c^2f'^3+3cf'f''+f''') \bar z_1^3,\nonumber
\end{equation}
\begin{equation}\label{ef4}
\frac{\partial^4 e^{cD_0}}{\partial z_1^4}=ce^{cD_0}(c^3f'^4+6c^2f'^2f''+3cf''^2+4cf'f'''+f^{iv}) \bar z_1^4.
\end{equation}
Plugging \eqref{fexpansion2} into \eqref{ef4} we get:
\begin{equation}
\begin{split}
\frac{\partial^8 e^{cD_0}}{\partial z_1^4\partial \bar z_1^4}|_0&=4!c(c^3f'(1)^4+6c^2f'(1)^2f''(1)+3cf''(1)^2+4cf'(1)f'''(1)+f^{iv}(1))\\
&=\frac{c}{3}(8c^3-24c^2+30c-15),
\end{split}\nonumber
\end{equation}
which is less than $0$ for any $0<c<a$, with $a=\frac{(4+4\sqrt5)^{1/3}}{4}-\frac{1}{(4+4\sqrt5)^{1/3}}+1>1$. 
\end{proof}

\begin{remark}\label{hardpart}\rm
We believe that $\left({\rm M}_{II}^{2n},cg_{II}\right)$ is not projectively induced also for any value of $c>1$. Let $D_0(x)=f(1+x)-f(1)=\sum_{j=1}^\infty a_j\frac{(\mathcal N-1)^j}{j!}$.
Accordingly with \eqref{n2} $a_1=1$, and for $j\geq 2$ a direct computation gives:
$$
a_{j}=\frac{d^j}{d\mathcal N^j}f(\mathcal N)|_{\mathcal N=1}=\frac{(-1)^{j-1}\prod_{s=1}^{j-1}(2s-1)}{2^{j-1}}.
$$
Thus:
$$
D_0(x)=\sum_{j=1}^\infty \frac{(-1)^{j-1}\prod_{s=1}^{j-1}(2s-1)}{2^{j-1}}\frac{(\mathcal N-1)^j}{j!},
$$
which implies:
$$
{\rm exp}\left(D_0(x)\right)-1=\sum_{n=1}^{\infty}\sum_{k=1}^nB_{n,k}(a_1,\dots, a_{n-k+1})\frac{(\mathcal N-1)^k}{k!},
$$
where $B_{n,k}(a_1,\dots, a_{n-k+1})$ are the exponential Bell polynomials (see \cite[p.133]{comtet}). Comparing with Qi and Guo \cite[p.8]{qiguo}, one gets that a sufficient condition for $cg_{II}$ to be not projectively induced is that for any $c>0$, there exists a big enough $n$ such that:
$$
\sum_{k=1}^nB_{n,k}(ca_1,\dots,c a_{n-k+1})=\frac1{2^{n}}\sum_{k=1}^n\frac{(-2c)^k}{k!}\sum_{l=0}^k(-1)^l{k\choose l}\prod_{q=0}^{n-1}(l-2q)<0.
$$
\end{remark}

\subsection{Complete Ricci--flat K\"ahler metric on the complexification of $\mathds C{\rm H}^n$}
Consider:
$$
{\rm M}_{III}^{4n}={\rm Gr}(2,2n+1,\mathds C)\setminus H_\infty,
$$ where ${\rm Gr}(2,2n+1,\mathds C)$ defines the complex Grassmanian of $2$--planes through the origin and, in homogeneous coordinates $(z,w)=(z_1,\dots, z_{2n+2},w_1,\dots, w_{2n+2})$, $H_\infty:=\left\{(z,w)|\, \sum_{j=1}^{n+1}\left(z_{2j}w_{2j-1}-z_{2j-1}w_{2j}\right)=0\right	\}$. Set inhomogeneous coordinates $(z_1,\dots, z_{2n},1,0,w_1,\dots, w_{2n},0,1)$ and denote $z=(z_1,\dots, z_{2n})$, $w=(w_1,\dots, w_{2n})$. Let:
\begin{equation}\label{n3}
\mathcal N:=\frac{(1+||z||^2)(1+||w||^2)-\sum_{j,k=1}^{2n}z_j\bar w_j\bar z_{k}w_{k}}{|\sum_{j=1}^{n}\left(z_{2j}w_{2j-1}-z_{2j-1}w_{2j}\right)-1|^2},
\end{equation}
and $f(\mathcal N)$ be a solution to:
\begin{equation}\label{equationmiii}
(2\mathcal N-1)\mathcal N^{2n-2}(f')^{4n}+2(\mathcal N-1)\mathcal N^{2n-1}(f')^{4n-1}f''=1.
\end{equation}
We can assume without loss of generality that $f(1)=0$. We have:
\begin{equation}\label{expansioniii}
\begin{split}
&f'(1)=1,\quad f''(1)=-\frac{n}{2n+1},\quad f'''(1)=\frac{6n^2+2n+1}{2(2n+1)^2},\\
 &f^{iv}(1)=-\frac{30n^3+22n^2+15n+2}{2(2n+1)^3}.
 \end{split}
\end{equation}
Observe that by \eqref{diast}, \eqref{n3} and since $f(1)=0$, we have:
$$
D_0(z)=f(\mathcal N)-f\left(\frac{1}{1-\sum_{j=1}^{n}\left(z_{2j}w_{2j-1}-z_{2j-1}w_{2j}\right)}\right)-f\left(\frac{1}{1-\sum_{j=1}^{n}\left(\bar z_{2j}\bar w_{2j-1}-\bar z_{2j-1}\bar w_{2j}\right)}\right).
$$
We are now in the position of concluding the proof of Theorem \ref{miith}.

\begin{proof}[Proof of Theorem \ref{miith} for $M^{4n}_{III}$]
We proceed as in the first part of the proof of Theorem \ref{miith}: we restrict ourselves to $z_2=\dots=z_{2n}=0$, $w_2=\dots=w_{2n}=0$, getting $D_0(z)=f(\mathcal N)=f(1+|z_1|^2)$. Thus, by \eqref{ef4} we have:
$$
\frac{\partial^8 e^{cD_0}}{\partial z_1^4\partial \bar z_1^4}|_{0}= 4!c(c^3f'(1)^4+6c^2f'(1)^2f''(1)+3cf''(1)^2+4cf'(1)f'''(1)+f^{iv}(1)),
$$
and as before, it is enough to prove that this quantity is negative for $0<c\leq1$ and $n\geq 2$. 
By \eqref{expansioniii} we then obtain:
\begin{equation}\label{finaleq}
\begin{split}
\frac{\partial^8 e^{cD_0}}{\partial z_1^4\partial \bar z_1^4}|_{0}=&4!\left(-\frac{6n}{2n+1}+\frac{15n^2+4n+2}{(2n+1)^2}-\frac{30n^3+22n^2+15n+2}{2(2n+1)^3+1}\right)\\
=&12c\left(2c^3-\frac{12n}{2n+1}c^2+\frac{2(15n^2+4n+2)}{(2n+1)^2}c-\frac{30n^3+22n^2+15n+2}{(2n+1)^3}\right),
\end{split}
\end{equation}
which is negative for any $n\geq 2$, $0<c\leq 1$. In fact, for $n=2$ it is equal to $
\frac{24}{25}c\left(25c^3-60c^2+70c-36\right)$, which is less than $0$ for $0<c<a$ with:
$$
a=\frac{(378+6\sqrt{11955})^{1/3}}{15}-\frac{22}{5(378+6\sqrt{11955})^{1/3}}+\frac45>1.
$$
Further, the RHS of \eqref{finaleq} is decreasing in $n$, since its derivative with respect to $n$ gives:
$$
\frac{24\left((-24c^2+44c-23)n^2+(-24c^2+14c+8)n-6c^2-4c-\frac32\right)}{(2n+1)^4},
$$
which, again, is negative since its numerator is negative and equal to $\frac{5}2(-60c^2+80c-31)$ when $n=2$, with derivative with respect to $n$ equal to
$$
-24(2n+1)c^2+2(44n+7)c-2(23n+4),
$$
and thus negative for any $c$ and any $n\geq 2$.
\end{proof}

\begin{remark}\rm For the case of the Stenzel's Ricci--flat K\"ahler metric $g$ on the cotangent bundle of $S^n$, identified with the affine quadric:
$$
Q^n:=\{(z_0,z_1,\dots, z_n)\in \mathds C^{n+1}|\ \sum_{j=0}^nz_j^2=1\},
$$
and with $\mathcal N$ defined as the restriction of $\sum_{j=0}^n|z_j|^2$ to $Q^n$, a K\"ahler potential  is given by $f(\mathcal N)$, where $f$ is a solution to:
\begin{equation}\label{Qn}
(f')^{n-1}\left(\mathcal Nf'+(\mathcal N^2-1)f''\right)=1.
\end{equation}
When $n=2$, $g$ is the Eguchi--Hanson metric, which has been proved in \cite{tohoku} to be not projectively induced even when rescaled by a positive factor. 

For $n>2$, with the same approach used in the proof of Theorem \ref{miith}, from \eqref{Qn} and \eqref{ef4} one gets that $cg$ is not projectively induced for any $c<\frac1{n+2}$. Although, the result can not be easily improved with similar techniques. 
\end{remark}

\end{document}